\documentclass[11pt]{amsart}
\usepackage{amsmath,amsthm,amscd,amssymb}
\usepackage{graphicx,hyperref}
\usepackage[all]{xy} 

\newtheorem{thm}{Theorem}[section]

\newtheorem{prop}[thm]{Proposition}

\newtheorem{lma}[thm]{Lemma}
\newtheorem{rem}[thm]{Remark}
\newenvironment{rmk}{\begin{rem}\rm}{\end{rem}}

\newtheorem*{thmstar}{Theorem}

\theoremstyle{definition}
\newtheorem{dfn}[thm]{Definition}

\theoremstyle{remark}
\newtheorem{remark}[thm]{Remark}
\newtheorem{ex}[thm]{Example}

\numberwithin{equation}{section}

\newcommand{\bbr}{\begin{remark}}        
\newcommand{\eer}{\end{remark}}

\font\bbb=msbm10 scaled 1100

\newcommand{\bea}{\begin{eqnarray}}
\newcommand{\eea}{\end{eqnarray}}
\newcommand{\bmini}{\begin{center}\begin{minipage}{5in}}
\newcommand{\emini}{\end{minipage}\end{center}}
\newcommand{\R}{{\mbox{\bbb R}}}
\newcommand{\C}{{\mbox{\bbb C}}}
\newcommand{\Z}{{\mbox{\bbb Z}}}

\newcommand{\nats}{\mbox{\bbb N}}

\newcommand{\A}{{\mathcal{A}}}
\newcommand{\M}{{\mathcal{M}}}

\newcommand{\pa}{\partial}

\newcommand{\img}{\operatorname{Im}}
\newcommand{\Spec}{\mbox{Spec}}
\newcommand{\Aug}{\mbox{Aug}}

\author[J.~Harper]{John G. Harper}

\author[M.~Sullivan]{Michael G. Sullivan}
\address{Department of Mathematics, University of Massachusetts,
 Amherst, MA 01003-9305}
\email{jgharper@student.umass.edu, sullivan@math.umass.edu}

\begin{document}

\title{A bordered Legendrian contact algebra}


\begin{abstract}
In \cite{Sivek11}, Sivek proves a ``van Kampen" decomposition theorem for the combinatorial Legendrian contact algebra (also known as the Chekanov-Eliashberg algebra) of knots in standard contact $\R^3$ .
We prove an analogous result for the holomorphic curve version of the 
Legendrian contact algebra of certain Legendrians submanifolds in standard contact $J^1(M).$ This includes all 1- and 2-dimensional Legendrians, and some
higher dimensional ones. We present various applications including a Mayer-Vietoris sequence for linearized contact homology similar to \cite{Sivek11} and a connect sum formula for the augmentation variety introduced in \cite{Ng08}. 
The main tool is the theory of gradient flow trees developed in \cite{Ekholm07}.

\end{abstract}

\maketitle

\section{Introduction}
\label{Introduction.sec}

Consider the one-jet space 
$J^1(N) = T^\ast N \times \R$
of an $n$-dimensional orientable manifold $N,$ equipped with the standard contact structure
$\ker\{dz - ydx\}.$ Here $z$ is the $\R$-coordinate and $(x,y) = (x_1, \ldots, x_n, y_1, \ldots, y_n)$ are local coordinates
for $T^*N.$  
Let $L \subset J^1(N)$ be a closed (possibly non-connected) $n$-dimensional orientable Legendrian submanifold.
For most of this paper, we consider the $n=2$ case; however, certain results
generalize to the $n=1$ case and a subset of the $n \ge 3$ case.

When $J^1(N)=  J^1(\R) = \R^3,$ Chekanov \cite{Chekanov02} and independently
Eliashberg \cite{Eliashberg00} developed a combinatorial differential
graded algebra (DGA), $Ch(L)$ for a Legendrian knot $L$ whose quasi-isomorphism type (which determines its homology) depends only on the Legendrian isotopy class. The homology of this Chekanov-Eliashberg algebra, while not a complete invariant,
distinguishes many Legendrian isotopy classes indistinguishable by previous invariants. The DGA, originally defined
over $\Z_2$-coefficients, was extended to $\Z H_1(L)$-coefficients in \cite{EtnyreNgSabloff02}.
At around the same time, a DGA for Legendrian submanifolds in any contact manifold,
known as the Legendrian contact algebra, $A(L),$
was sketched as part of symplectic field theory \cite{EliashbergGiventalHofer00}.
This theory is defined using $J$-holomorphic (pseudo-holomorphic) disks
and generalizes the Chekanov-Eliashberg algebra. 
In the $L \subset J^1(N)$ case, $A(L)$ was shown to be a well-defined theory in \cite{EkholmEtnyreSullivan07}. 

Motivated by a decomposition theorem for Heegaard Floer homology 
\cite{LipshitzOzsvathThurston}, Sivek  in \cite{Sivek11} proves a decomposition
theorem for the front reformulation of the Chekanov-Eliashberg algebra introduced
in \cite{Ng03}.
He divides the algebra's generators into two sets, based on which side
of an arbitrary $x=c$ line in the front projection $J_0(\R) = \R^2_{x,z},$
the generators' geometric representatives lie.
Introducing a new set of generators associated to the dividing line, he constructs
three DGAs: $I_n,$ generated by this new set; $D(L^D),$ whose generators include
those on the right (greater $x$-value) of the line; and $A(L^A),$ whose generators include those on the left. He shows there is a natural commutative diagram
\[
\xymatrix{
I \ar[d] \ar[r] &D(L^D)\ar[d]\\
A(L^A)\ar[r] &Ch(L) }
\]
and proves it is a push-out square in the category of DGAs.

The Legendrian contact algebra $A(L)$ is generated by the set of Reeb chords of the Legendrian submanifold. The differential counts certain $J$-holomorphic disks.
Consider any decomposition of these generators into two sets, $S_1$ and $S_2,$ which analogous to \cite{Sivek11}, corresponds to a decomposition
of $L$ into $L_1$ and $L_2$ by some (possibly disconnected) hypersurface in $J^0(N) = N \times \R_z$ containing the $z$-direction. $S_i$ are the chords of $L_i.$
Through a family of Legendrians, we deform $L$ to $L' = L'_1 \cup L'_2$ in a neighborhood of this hypersurface, possibly generating new Reeb chords in a controlled way (see Remark \ref{VKrefined.rem}). Here $L_i'$ is a small thickening in $L'$ of the image of $L_i$ under the isotopy, and so $L_1'$ will intersect $L_2'.$
Note that $A(L)$ and $A(L')$ are quasi-isomorphic, and $A(L)$ sits in $A(L')$ as a graded sub-algebra, but not necessarily as a sub-DGA.
We construct three DGAs: 
$A_1$ generated  by the chords of $L_1';$
$A_2$ generated  by the chords of $L_2';$ and
$A_3$ generated by the chords of $L_1' \cap L_2'.$
All 3 are sub-DGAs of $A(L')$ and $A_3$ is a sub-DGA of $A_1$ and of $A_2.$

The key step is to substitute the theory of $J$-holomorphic disks with the theory of gradient flow trees, which was developed in \cite{Ekholm07} and is much easier to manage. See \cite[Theorem 1.1]{Ekholm07} (restated as \cite[Theorem 2.5]{EkholmKalman08} and \cite[Theorem 2.10]{Rizell}).
Not only is there a one-to-one correspondence, but the boundaries of the disks
``match" with the gradient flow trees.
Whenever the hypotheses of \cite[Theorem 1.1]{Ekholm07}
hold (which include all Legendrians $L$ in $J^1(N)$ when $n=1,2$), we prove the following ``van Kampen" theorem.

\begin{thm} \label{VK.thm}
The DGA inclusion maps form a commutative square
\[
\xymatrix{
A_3 \ar[d] \ar[r] &A_1 \ar[d]\\
A_2 \ar[r] &A(L') }
\]
which is a pushout square in the category of DGAs.
\end{thm}
This result applies to whichever coefficient ring for which $A(L)$ (and hence $A(L')$) is defined.

After surveying some contact geometry and homological algebra in Section \ref{Background.sec},
we state and prove Theorem \ref{VK.thm} and a generalization, Theorem \ref{VKmultiple.thm}, in Section \ref{Main.sec}. 
When $J^1(N) = \R^3,$ our pushout square can differ from the one in \cite{Sivek11}.

In Section \ref{Applications.sec}, we present some applications and computations of Theorem \ref{VK.thm}.
We begin with a Mayer-Vietoris sequence for all the linearized contact homology theories
of $L.$ This is exactly as in \cite{Sivek11}, and in fact our result is easier.
Note that since a DGA is considered as a ``homotopy theory," and the linearization the
``corresponding homology," a Van Kampen theorem (Theorem \ref{VK.thm}) and a
Mayer-Vietoris sequence are not surprising.
We next consider the cusp connected sum construction $\#$ introduced
in \cite{EkholmEtnyreSullivan05a} for Legendrian submanifolds of dimension 
$n \ge 2.$ 
We relate the augmentation varieties (see Section \ref{Background.sec}) of the Legendrians and their sum.
\begin{thmstar} [\ref{CS.thm}]
\[
\Aug(A_{\mathbb{C}H_1(L_1)},\partial_{L_1})\times \Aug(A_{\mathbb{C}H_1(L_2)},\partial_{L_2})= \Aug(A_{\mathbb{C}H_1(L)},\partial_L).
\]
\end{thmstar}
We end Section \ref{Applications.sec} with some examples.

 {\em Acknowledgements :}
The authors were partially supported by NSF grant DMS-1007260.
They wish to thank discussions with Paul Hacking and Steven Sivek, useful detailed feedback from the referee, as well as
the UMass Research Experience for Undergraduates program.


\section{Background} \label{Background.sec}

In Section \ref{LCA.sec}, we briefly review the $J$-holomorphic disk definition
of the Legendrian contact algebra defined in \cite{EkholmEtnyreSullivan05a, EkholmEtnyreSullivan05c, EkholmEtnyreSullivan07}.
In Section \ref{GFT.sec}, we recall the gradient flow tree reformulation
of this algebra from \cite{Ekholm07}.
In Section \ref{Augmentations.sec}, we review the augmentation variety defined
in \cite{Ng08} and the linearized contact homology theories for a Legendrian
first introduced in \cite{Chekanov02}.

\subsection{Legendrian contact algebra} \label{LCA.sec}

Recall the standard contact structure of $J^1(N)$ has a canonical
contact form, locally written as $dz - ydx := dz - \sum_i^n y_i dx_i.$
Since the Reeb vector field is $\pa_z,$ the orbits starting and ending on the Legendrian submanifold $L,$ known as Reeb chords, are vertical.

Consider the Lagrangian and front projections
\[
\pi_{\mathbb C}: J^1(N) = T^\ast N \times \R \to T^\ast N \quad \mbox{and}\,\, \pi_F: J^1(N) \to J^0(N) = N \times \R.
\]
For generic $L,$ $\pi_{\mathbb C}(L)$ is an exact Lagrangian immersion with isolated
transverse double points, and $\pi_F(L)$ is the graph of local Morse-Smale 
functions $f_i.$
The set of double points is canonically isomorphic to the set of critical points of positive
difference functions $f_i-f_j,$ which we can assume to be Morse-Smale as well.
This set is also isomorphic to the set of Reeb chords of $L.$ 

We grade this set.
Given a Reeb chord $c$ associated to a critical point $x$ of $f_i - f_j,$ let 
$c^\pm$ denote the endpoints of the vertical line segment $\pi_F(c),$ with $c^+$ having higher $z$-coordinate. Assume $c^+$ and $c^-$ lie in the same component of $L.$
Pick a path $\gamma_c: I \to \pi_F(L)$ from $c^+$ to $c^-$ that transversely intersects the
codimension 1 cusp singularities of $\pi_F(L),$ and avoids all higher codimension front
singularities. 
The path $\gamma_c$ is called a capping path for $c,$ in the literature.
Let $D(\gamma_c)$ and $U(\gamma_c)$ denote the number
of cusp edges $\gamma_c$ traverses in the negative and positive $z$-direction.
Define its grading to be
\[
|c| = D(\gamma_c) - U(\gamma_c) + \mbox{ind}_x (f_i - f_j) - 1.
\]
Here $\mbox{ind}$ denotes the Morse index.
This grading is well-defined as an element of $\Z /\{\mbox{Im} (\mu : H_1(L) \to \Z)\}$ where $\mu$ is the Maslov class, see \cite{EkholmEtnyreSullivan05a} for details.
If $c^+$ and $c^-$ lie in different components of $L,$ 
we modify the construction of $\gamma_c$ in a certain way to obtain a $\Z_2$ grading as done in Section 2.3 of \cite{EkholmEtnyreNgSullivan11}, for example.

Next we describe the graded algebra.
Fix a (possibly trivial) subgroup $G \subset H_1(L).$
Let $R$ be the group ring $\Z[G]$ or $\C[G]$ if $L$ is spin, and $\Z_2[G]$ otherwise.
Let $A(L)$ be the unital tensor algebra freely generated by the set
of Reeb chords, over the ring $R.$ For $A \in H_1(L)$ and
the (possibly empty) ordered word ${\mathbf b} :=b_1 \cdots b_m$ define the grading
$|A \mathbf{b}| = \mu(A) + \sum |b_i|.$

Finally we describe the differential.
Fix an almost complex structure $J$ compatible with the standard symplectic structure
on $T^\ast N.$ Choose a Reeb chord $a,$ a (possibly empty) word of chords
${\mathbf b}$ and a homology class $A \in H_1(L).$
Let $D \subset \C$ be the unit disk with $m+1$ marked boundary points $p_0 = 1,
p_1 = i, p_2 = -1, p_3, \ldots, p_{m}$ listed in counter-clockwise (positive) order.
Define the moduli space ${\M}_A(a;{\mathbf{b}})$ to be the set of $J$-holomorphic
disks  $u: D \to T^*M$ such that 
\[
u(\pa D) \subset \pi_{\mathbb C}(L), \,\,u(p_0) = \pi_{\mathbb C}(a), \,\,u(p_i) = \pi_{\mathbb C}(b_i), \,\,[\widehat{u(\pa D)}] = A,
\]
where $u|_{\pa D}$ near $p_0$ (resp. $p_i$ for $i \ge $1) in the positive direction maps from the lower (resp. upper) branch of $\pi_{\mathbb C}(L)$ at $\pi_{\mathbb C}(a)$ (resp. $\pi_{\mathbb C}(b_i)$)
to the upper (resp. lower) branch of $\pi_{\mathbb C}(L),$
and where $\widehat{u(\pa D)}$ is the image of the boundary, lifted to $L,$ and closed off 
at $a$ (resp. $b_i$) by the path $\gamma_a$  (resp. $\gamma_{b_i}^{-1}$).
In \cite{EkholmEtnyreSullivan05b, EkholmEtnyreSullivan05c}
it is shown that for generic $J$ or perturbation of $L$, ${\M}_A(a;\mathbf{b})$, modulo
conformal reparameterizations, is a compact (in the sense of Gromov) manifold of dimension $|a| - |A\mathbf{b}| - 1.$
If $L$ is spin, then ${\M}_A(a;\mathbf{b})$ can be equipped with a coherent (that is,
compatible with gluing) orientation \cite{EkholmEtnyreSullivan07}.
Define the differential $\pa: A(L) \to A(L)$ by
\begin{equation}\label{stddiff}
        \partial a =
        \sum_{\{u \in {\M}_A(a;{\mathbf b}) \,\,|\,\,
        |a|-|\mathbf{b}|-|A| -1=0\}}
        (-1)^{|a|+1}
        \sigma(u)  A {\mathbf b},
\end{equation}
where $\sigma(u) =\pm1 \in R$ is determined by the moduli
space orientation. The differential is then extended to all of
$A(L)$ by the graded product rule and linearity. 

In \cite{EkholmEtnyreSullivan05b, EkholmEtnyreSullivan05c, EkholmEtnyreSullivan07},
it is shown that this differential squares to zero, and the DGA is invariant under
stable-tame isomorphisms; thus the homology is invariant.
 See \cite{Chekanov02} for a definition of stable-tame isomorphisms.

\subsection{Gradient flow trees} \label{GFT.sec}
We briefly review the results of \cite{Ekholm07}, also reviewed in 
\cite[Section 2.4]{EkholmEtnyreNgSullivan11}, \cite[Section 2.3]{EkholmKalman08} and \cite[Section 2.2]{Rizell}.

\begin{dfn}
Pick a metric $g$ on $N.$ A {\em gradient flow tree} on $L \subset J^1(N)$ is an immersion  
$\phi: \Gamma \to N$ of a finite (not necessarily planar) tree $\Gamma$ with the following
data and conditions. 
\begin{itemize}
\item
At every edge $e_k$ of $\Gamma,$ $\phi$ parameterizes some gradient flow line
of $-\nabla(f_i - f_j)$ where the graphs of $f_i>f_j$ in $J^0(N)$ locally model
$\pi_F(L)$ as in Section \ref{LCA.sec}. Orient the two 1-jet lifts (under the projection $J^1(N) \to N$) $\phi_k^1$ and $\phi_k^2$ in $L$  of $\phi(e_k)$ according
to the vectors $-\nabla(f_i - f_j)$ and $+\nabla(f_i - f_j).$
\item
Every $l$-valent vertex $v$ comes equipped with a cyclic ordering of the edges
which we temporarily denote $e_1, \ldots, e_l, e_{l+1} =e_1.$
The Lagrangian projection of their 1-jet lifts has the following compatibility:
$\pi_{\mathbb C}(\phi^2_k(v)) = \pi_{\mathbb C}(\phi^1_{k+1}(v))$ where $v$ is viewed as a point
of both edges $e_k$ and $e_{k+1}.$ 
\item
Concatenating the Lagrangian projection of these 1-jet lifts in such a manner at each vertex produces a closed curve in $\pi_{\mathbb C}(L).$ 
\end{itemize}
\end{dfn}

If the $z$-coordinate of the 1-jet lift $\phi^2_k(v)$ is greater (resp. lesser) 
than the $z$-coordinate of $\phi^1_{k+1}(v),$ we say the flow tree has a
{\em positive} (resp. {\em negative}) {\em puncture} at the vertex $v.$
These punctures occur at Morse critical points of $f_i-f_j$ (identified with Reeb chords).

Proposition 3.14 of \cite{Ekholm07} establishes that, for generic perturbation of $L$ and choice of metric $g,$ the set of gradient flow trees with one positive puncture at $a$
and negative punctures at 
the (possibly empty) ordered word ${\mathbf b}$, is a
manifold whose dimension, as in Section \ref{LCA.sec}, is $
|a| - |A\mathbf{b}| - 1.$ 
Here $A \in H_1(L)$ is represented by the union of the 1-jet lifts of edges, closed off at $a$ with the capping path $\gamma_a$ and at $b_i$ with the inverse capping path $\gamma_{b_i}^{-1},$ as done in Section \ref{LCA.sec}.

We say a $J$-holomorphic disk or a gradient flow tree is {\em rigid} if its (expected)
dimension is 0.
Assume that $L$ is 1 or 2 dimensional, or that $\pi_F(L)$ has at most only cusp-edge
singularities.
Theorem 1.1 of \cite{Ekholm07} proves that for a generic perturbation of $L$
(see \cite[Section 2.2.1]{Ekholm07} for gradient flow behavior near cusp-edges),
there is a regular compatible almost complex structure $J$ such that the rigid
$J$-holomorphic disks are in one-to-one correspondence with the rigid gradient flow trees. Moreover, the boundary of each disk is $C^0$-close to the Lagrangian
projection of the 1-jet lifts of the corresponding gradient flow tree.

Rigid gradient flow trees of  Legendrians which satisfy the above hypotheses have a manageable set of local models.
By \cite[Remark 3.8]{Ekholm07}, the vertices must be 
\begin{itemize}
\item
valence 1, at a puncture or transverse to a cusp-edge;
\item 
valence 2, at a puncture or tangent to a cusp-edge; or,
\item 
valence 3, away from or transverse to a cusp-edge.
\end{itemize}
See \cite[Section 2.4]{EkholmEtnyreNgSullivan11} for figures of these. 

\subsection{Augmentations} \label{Augmentations.sec}

We recall some algebraic geometry needed to discuss augmentation varieties. Let $R$ be a finitely generated commutative  $\mathbb{C}$-algebra, and let $\Spec(R)$ be the set of all maximal ideals of $R.$ Consider the zero locus $Z(f)$ of an element $f\in R$ on $\Spec(R)$ in the following way: 
\[
Z(f)=\left\{m \,\, \left| \,\, m\in \Spec(R), f\in m \right\} \right. \subset \Spec(R).
\] 
Equivalently, the residue class of $f$ vanishes in the residue field $R/m.$ This definition immediately extends to ideals: 
\[Z(I)=\{m\,\,|\,\,m\in \Spec(R),I \subset m \} .
\] 
$\Spec(R)$ naturally carries the Zariski topology where all closed sets are of the form $A=Z(I)$ where $I\subset R$ is an ideal. 

If $R=\mathbb{C}[x_1,..,x_n]$, then the \emph{Nullstellensatz} tells us that every maximal ideal of $R$ is of the form 
\[
m_p = \langle x_1 - p_1, x_2 - p_2,\ldots,x_n - p_n\rangle=\ker(ev_p)
\] 
where $p\in\mathbb{C}^n$ and $ev_p: \mathbb{C}[x_1,..,x_n]\rightarrow \mathbb{C}$ is the evaluation homomorphism at $p.$ 
In this case, we can naturally identify $\Spec(R)$  with $\mathbb{C}^n.$ Under this identification, note that $f\in R$ vanishes in the residue field $R/m_p$ iff $f(p)=0.$ Hence, we see that $Z(I)$ is simply the zero locus in $\mathbb{C}^n$ of any collection of generators for $I.$  Recall that an ideal $J\subset R$ is said to be \emph{reduced} if $J=\mbox{Rad}(J)$ where 
\[
\mbox{Rad}(J)=\{f\in R| f^n \in J\ \,\, \mbox{for some}\,\,  n\in \mathbb{N}\}.
\]
The Nullstellensatz tells us that there is a natural bijection between closed 
(in the Zariski topology) sets in $\Spec(R)$ and reduced ideals in R. Moreover, $\Spec(R/J)$  is homeomorphic to $Z(J)\subset R.$ 

This relationship described above turns out to be functorial. Given a morphism 
$g:R\rightarrow S$ of finitely generated $\mathbb{C}$-algebras, there is a induced map $g^*:\Spec(S)\rightarrow \Spec(R)$ defined as follows. Given $m_p \in \Spec(S)$, let $\pi_{m_p}:S\rightarrow S/m_p$ be the canonical map, then $\pi_{m_p}\circ g$ is surjective since $g$ is a morphism of $\C$-algebras, so it has kernel $m_q \subset R$ where $m_q \in \Spec(R).$ We define $g^* (m_p)=m_q.$ For a complete exposition of the above discussion see \cite{Eisenbud}.

Given a DGA $(A,\partial)$ over a ring $R,$ an {\em augmentation} is a unital morphism $\epsilon: (A,\partial)\rightarrow (R,\pa = 0)$ satisfying the following property: first define a tame automorphism $\sigma_{\epsilon}:A \rightarrow A$ by $c \mapsto c+\epsilon (c),$ then the DGA $(A,\partial^\epsilon := \sigma_\epsilon \partial \sigma_\epsilon^{-1})$ satisfies the property that the only constant term (element of degree $0$) in the image of $\partial^\epsilon$ is $0\in A.$ If such an augmentation exists, we say that $(A,\partial)$ is {\em good}.

In general, augmentations allow us to consider a simplified homology theory. Given an augmentation $\epsilon:(A,\partial)\rightarrow (R,0)$, we may decompose the differential $\partial^\epsilon$ into $\oplus_{j=0}^{\infty} \partial_j^\epsilon$   where  $\partial_j^\epsilon(x)$ contains all words of length $j$ in the image of $\partial^\epsilon(x)$. Since $\partial_0^\epsilon=0$, it follows that $(\partial_1^\epsilon)^2 = 0$. The {\em linearized contact homology} is then defined to be the graded vector space $LCH(A,\partial,\epsilon)=\ker(\partial_1^\epsilon)/ \img(\partial_1^\epsilon) .$ The corresponding {\em Poincar\'e polynomial} (in the formal variable $t$) is defined to be 
\[
P_{A, \partial, \epsilon}(t) = \sum  \dim (LCH_i(A,\partial,\epsilon)) t^i .
\]

Let $L$ be a Legendrian submanifold which is spin and has vanishing Maslov class.
Let $(A,\partial)$
be the  Legendrian contact algebra
of   $L$ with coefficients in $\C[H_1(L)].$
Any $\mathbb{C}$-algebra morphism $p:\mathbb{C}H_1(L)\rightarrow \mathbb{C}$ extends canonically to a unital DGA map 
\[
p:(A,\partial) \rightarrow (pA,p\partial).
\]

\begin{dfn}
Let $X$ be the set of morphisms $p: \mathbb{C}H_1(L)\rightarrow \mathbb{C}$ such
that $(pA, p \partial)$ is good. 
$X$ induces a collection of {\em good points} $X'\subset \Spec(\mathbb{C}H_1(L))$. 
The {\em augmentation variety} $\Aug(\mathbb{C}H_1(L)\otimes A,\partial)$ is defined as the Zariski closure of $X'$ in $ \Spec(\mathbb{C}H_1(L))$.
\end{dfn}

See \cite{Ng08} for the original definition of, as well as applications of  the augmentation variety.


\section{Main result} \label{Main.sec}
In this section we prove Theorem \ref{VK.thm}, as well as state a generalization and a refinement.

\subsection{Proof of Theorem \ref{VK.thm}} \label{VKProof.sec}

Fix a metric on $N$ which induces one on $J^1(N)$ such that \cite[Theorem 1.1]{Ekholm07} holds. 

Recall we have a Legendrian $L \subset J^1(N)$ whose Reeb chords are
divided into two sets $S_1$ and $S_2.$
Let $h \subset N$ be a (possibly disconnected) hypersurface such that $N \setminus h$
is at least two components, and the projection $J^1(N) \to N$ of the chords
sends no two chords from different sets to the same component. 
Write $N = N_1 \cup N_2$ as the union of two manifolds with boundary 
such that $\pa N_1 = \pa N_2 \subset h$ and chords
in $S_i$ project to $N_i.$

Let $H  = \pi_F^{-1}(h)$ be the hypersurface in $J^0(N)$ which
divides $\pi_F(L)$ into two fronts $F_1$ and $F_2,$ sitting over $N_1$ and $N_2$ respectively, with boundary in $H.$
For generic choice of $h,$ we can assume $H$ intersects
$\pi_F(L)$ at its cusps transversely, and in the dimension $n=2$ case, avoids any 
dovetail singularities. 
(By hypothesis, for dimensions $n \ge 3$ there are no
other singularities.) For $i=1,2$ and $0 < \epsilon \ll 1,$ 
define the set $L_i$ open in $L$ such that $F_i \subset \pi_F(L_i)$ is a deformation retract of $\pi_F(L_i),$ and $\pi_F(L_i)$ is contained in (resp. contains)
a $2\epsilon/3$ (resp. $\epsilon/3$) neighborhood of $F_i \subset \pi_F(L).$ 

For a Reeb chord $c,$ its {\em action} $\A(c)$ is defined as the (positive) difference in $z$-coordinates of endpoints $c^+$ and $c^-.$
Fix $0 < \delta \ll \min_{c} \A(c)$ where the minimum is taken over all 
Reeb chords of $L.$ (Since $L$ is compact, we can assume all actions are of ``order 1.")
Choose a smooth even ``inverted bump" function $g = g_\delta: [-\epsilon, \epsilon] \to [\delta,1]$ with the following properties
\begin{itemize}
\item
$g(\pm \epsilon) = 1$ and $g$ has a unique critical point (minimum)
$g(0) = \delta;$  
\item
$\delta \le g|_{[-\epsilon/3,\epsilon/3]} \le 2 \delta;$
\item
$1 - \delta \le g|_{(-\epsilon,-2\epsilon/3] \cup [2\epsilon/3, \epsilon)} <1;$
\item
$g|_{[-2\epsilon/3,-\epsilon/3] \cup [\epsilon/3, 2\epsilon/3]}$ is linear.
\end{itemize}

Choose local coordinates of $(t,{\bf s}) \in[-\epsilon, \epsilon] \times h$ of a neighborhood of $h$ in $N$
such that the $(n-1 \ge 0)$ ${\bf s}$-coordinates are tangent to $h.$
Choose $\epsilon$ and $\delta$ such that for $\epsilon/3 < t < 2\epsilon/3$
$|g'(t)|$ is big enough to ensure
\[ 
\left| \frac{\pa}{\pa t} (g (f_i-f_j)) \right| > 0 
\]
outside of an $\min(\delta, \epsilon/3)$-neighborhood of a cusp-edge.

Define a base-preserving isotopy of $J^0(N) = N \times \R$ constant away from 
$[-\epsilon, \epsilon] \times H$ by
\begin{equation} \label{squeeze.eq}
\phi_\lambda(t,{\bf s},z) =  (t, {\bf s}, (1-\lambda)z + \lambda g(t) z)
\quad \mbox{for} \,\, 0 \le \lambda  \le 1.
\end{equation}
Note that this ambient isotopy of $J^0(N)$ induces a 
Legendrian isotopy $L^\lambda \subset J^1(N)$ by requiring 
$\pi_F(L^\lambda) = \phi_\lambda(\pi_F(L))$ and $L^0 =L.$
Furthermore, all new (possibly degenerate) Reeb chords must have $N$-coordinates
in $[-\epsilon, \epsilon] \times h,$ and be within $\epsilon/3$ of $h,$
a cusp-edge, or the boundary of support of the isotopy.

For technical reasons, we require the Legendrians to remain ``front generic" in the 
sense of \cite[Section 2.2.1]{Ekholm07}, which means the gradient flows
intersect cusps transversely. As defined $L^\lambda$ for $\lambda >0$ may 
have gradient flows vanish where any cusp intersects $H.$
To fix this, choose $0 < \delta' \ll \min(\delta, \epsilon/3),$ and deform $L^\lambda$
(continuously in $\lambda$) in a $\delta'$-neighborhood of any portion of
those cusps which lie over $[-\epsilon, \epsilon] \times h$ as described in
Figure \ref{CuspMin.fig}.

\begin{figure}
 \centering
  \includegraphics[height=2in]{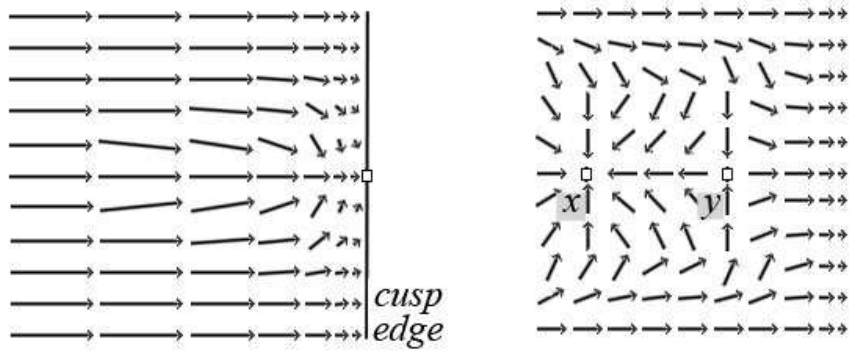}
\caption{This is drawn in the dimension $n=2$ case. Note how the boundary min
is in fact a minimum $x$ and a saddle $y.$ The vector field denotes $-\nabla (f_1 - f_2)$
where $f_1$ and $f_2$ correspond to the two sheets defining the cusp.}
\label{CuspMin.fig}
\end{figure}

The Reeb chords whose action is of order $\delta$ are exactly those whose $N$-coordinates lie in an $\epsilon/3$ neighborhood of $h \subset N.$ 
(This includes the new chords as drawn in Figure \ref{CuspMin.fig}.)
As mentioned above, there may also be chords  whose base coordinates lie outside a $2\epsilon/3$-neighborhood but inside an $\epsilon$-neighborhood of $h.$ 
Their actions are order 1.
The remaining set of chords have a canonical one-to-one correspondence with the chords of the original $L$ since the isotopy is fixed away from those points lying above $h.$

Since we want the Reeb chords to be transverse, we perturb our Legendrian
by a $C^2$-small isotopy of order $0 < \delta'' \ll \delta'.$ 
This does not create any new Reeb chords, nor change the actions
(up to order $\delta''$) of the chords that persist.
This Legendrian will be our $L'$ from the statement of the Theorem.
Define the open decomposition $L' = L_1' \cup L_2'$ as was done
for $L = L_1 \cup L_2$ above. 
To be safe, we will also include to both $L_1'$ and $L_2'$
an $\epsilon/3$-neighborhood of any portion of cusp edges 
in the $[-\epsilon, \epsilon] \times h$ region.

For $i=1,2,$ let $A_i$ denote the algebras generated by the chords of $L_i'.$
Let $A_3$ denote the algebra generated by any chord with action bounded
by order $\delta.$ Since the differential decreases the action of the generators
(by Stokes' Theorem), $A_3$ is a sub-DGA of $A(L'),$ as well as of $A_1$ and $A_2.$ 

It suffices to show that $A_1$ and $A_2$ are also sub-DGAs of $A(L').$
Suppose $a$ is a chord of $L_1'$ and $\phi : \Gamma \to N$ is a rigid flow tree which intersects $h$ at $p \in \Gamma,$ contributing the word $w;$ that is, 
$\pa  a = w + \ldots.$  
The one-jet lifts, $\phi^1(p)$ and $\phi^2(p),$ of the flow at that point have
$z$-coordinates which differ by order $\delta.$ Since the difference in $z$-coordinates
must decrease and stay positive, and since $\pi_F(L')$ is front-generic in the
sense of Figure \ref{CuspMin.fig}, the flow cannot leave the $\epsilon/3$-neighborhood
of $h$ or the cusp-edge in $[-\epsilon, \epsilon] \times h.$ In particular, it contributes a (possibly constant) word in $A_3$ to $w.$

\subsection{A generalization and a refinement}

We state without proof a generalization of Theorem \ref{VK.thm}.
This is a direct analogy to \cite[Theorem 2.20]{Sivek11}, and the proof is a slight
modification of the one in Section \ref{VKProof.sec}.

Let $L_1,L_2,L_3 \subset L $ be three $n$-dimensional submanifolds with boundary such that $L_1\cup L_2\cup L_3 =L.$ 
Assume that $L_1\cap L_2$ is empty and that 
$\partial (L_3) = \partial(L_1) \cup \partial(L_2).$
Apply the construction of Section \ref{VKProof.sec} locally along the boundary of $L_3$ to obtain the following generalization of Theorem \ref{VK.thm}.
Let $A_3, A_{13}, A_{23}, A$ be the Legendrian contact algebra of 
$L_3', L_1' \cup L_3', L_2' \cup L_3', L_1' \cup L_2' \cup L_3'$
respectively.
\begin{thm} \label{VKmultiple.thm}
The following is a push-out square
\[
\xymatrix{
A_3 \ar[d] \ar[r] &A_{13} \ar[d]\\
A_{23} \ar[r] &A }
\]
\end{thm}

\begin{rmk} \label{VKrefined.rem}
The argument in Section \ref{VKProof.sec} can be made with more care
to control the number of new Reeb chords which appear. Specifically,
consider the intersection of the dividing hypersurface $H \subset J^0(N)$
with the Legendrian front $\pi_F(L)$ as a graph over a $(n-1)$-manifold $h.$
Each critical point of the difference function of pairs of sheets of $h$ leads
to two new Reeb chords (of relative index 1) of our modified Legendrian.
Furthermore, each cusp edge leads to new Reeb chords (of relative index 1)
as shown in Figure \ref{CuspMin.fig}.
To prove this, however, requires some manipulations of the gradient flows, and
we have no application of this refinement.
\end{rmk}


\section{Applications and computations} \label{Applications.sec}

\subsection{Mayer-Vietoris and characteristic algebra}
With Theorem \ref{VK.thm} in hand, it is natural to ask about its consequences on the standard constructions of Legendrian contact homology. Following Sivek \cite{Sivek11}, we show that Theorem \ref{VK.thm} induces a Mayer-Vietoris sequence on linearized contact homology. Another standard functorial construction, the characteristic algebra behaves even more nicely, in the sense that it preserves the pushout square.

By Theorem \ref{VK.thm}, we have the push-out square

\begin{equation}
\label{PushOutAugmentation.eq}
\xymatrix{
&A_1\ar@{.>}[dr]\ar[drr]^{\epsilon_1} && \\
A_3\ar[ur]\ar[dr] &&A\ar[r]|\epsilon&\mathbb{C} \\
&A_2\ar@{.>}[ur]\ar[urr]_{\epsilon_2}&&
}
\end{equation}
where the unmarked arrows are inclusion maps. So any augmentation of $A$ induces augmentations $\epsilon_i$ on $A_i.$  Conversely, given any augmentations $\epsilon_1$ and $\epsilon_2$ commuting as above (hence which define 
$\epsilon_3: A_3 \to \C$), we have an induced augmentation on $A$ by the push-out property. 

\begin{prop} \label{LCH.prop}

There is a long exact sequence
\begin{eqnarray*}
&\ldots \rightarrow LCH_*(A_3,\partial,\epsilon_3)\rightarrow LCH_*(A_1,\partial,\epsilon_1)\oplus LCH_*(A_2,\partial,\epsilon_2) \rightarrow &
\\
&LCH_*(A,\partial,\epsilon)\rightarrow \ldots &
\end{eqnarray*}
of linearized homology groups.
\end{prop}
\begin{proof} Define
\begin{eqnarray*}
f: LCH(A_3,\partial,\epsilon_3)\rightarrow LCH(A_1,\partial,\epsilon_1)\oplus LCH(A_2,\partial,\epsilon_2), & & f(x) = (x,-x), \\
g: LCH(A_1,\partial,\epsilon_1)\oplus LCH(A_2,\partial,\epsilon_2) \rightarrow LCH(A_3,\partial,\epsilon_3),
& & g(x,y) = x+y.
\end{eqnarray*}
Clearly $g \circ f = 0$, so $\img(f) \subset \ker(g).$
To see that $\ker(g) \subset \img(f)$, it suffices to note that if $(x,y)\in\ker(g)$, then $x=-y$, so
$x\in LCH(A_1,\partial,\epsilon_1)$ and $x\in LCH(A_2,\partial,\epsilon_2).$
By the description of the generators preceding Theorem \ref{VK.thm}, it is clear that any generator contained in both $A_1 $ and $A_2$ must also be contained in $A_3$, whence we have
$x\in LCH(A_3,\partial,\epsilon_3).$
\end{proof}

Ng \cite{Ng03} defines the characteristic algebra for the Chekanov-Eliashberg algebra
of Legendrian knots, but his construction-definition extends to the Legendrian
contact algebra of higher dimensional Legendrian since the invariance
proof of the theory uses the same set of stable-tame isomorphisms \cite{EkholmEtnyreSullivan05b}.

\begin{dfn}\label{CA.dfn}
Let $(A,\partial)$ be a DGA. 
The {\em characteristic algebra} $C(A,\partial)$ is defined as $A/I$ where $I$ is the two-sided ideal generated by $\partial(A).$
\end{dfn}

Given a morphism $f:(A,\partial)\rightarrow(B,\partial')$ of DGAs, we see that $f\partial(A)=\partial' f(A) \subset\partial ' (B)$. So, $f$ induces a morphism $C(A)\rightarrow C(B)$. Thus, we see that $C$ is a functor from the category of DGAs to the category of unital associative graded algebras.

\begin{prop} \label{CAFunctor.prop}

\cite{Sivek11} The functor C preserves pushouts.

\end{prop}


\subsection{Poincar\'e Polynomials, Augmentation varieties and cusp connected sums}

In this subsection we prove a formula computing the augmentation variety of a certain connect sums of Legendrian submanifolds, and speculate on the corresponding formula for the Poincar\'e polynomial.

Recall that if $L$ is a compact orientable manifold with $\dim(L)\geq2$ and $L=L_1\#L_2$ is the connect-sum of
two closed manifolds $L_1,L_2,$ then $H_1(L) \cong H_1(L_1) \oplus H_1(L_2).$ Thus, $\C[H_1(L)]\cong\C[H_1(L_1)]\otimes_{\C} \C[H_1(L_2)]$ under the obvious identification and we have a push-out square of commutative $\C$-algebras:
\begin{equation}
\label{PushOutCoef.eq}
\xymatrix{
\C \ar[d] \ar[r] &\C[H_1(L_1)] \ar[d]^{\rho_1}\\
\C[H_1(L_2)] \ar[r]^{\rho_2} &\C[H_1(L)] }
\end{equation}

We recall a construction from \cite[Section 4.2]{EkholmEtnyreSullivan05a}.
Given two Legendrian submanifolds $L_1, L_2 \subset J^1(N),$ their cusp connected sum
$L_1 \# L_2,$ can be (non-uniquely) defined by first applying a Legendrian isotopy (if it exists) so that the fronts $\pi_F(L_1),\pi_F(L_2)$ are separated by a hyperplane in $N \times \R$ containing the $z$-direction. 
Let $c \subset J^1(N)$ be an arc beginning at a cusp-edge of $L_1,$ ending at a cusp-edge of $L_2,$ and parameterized by 
$s \in [-1,1].$ 
Choose a neighborhood $M$ of $c$ whose vertical cross sections consist of round balls whose radii vary with $s$ and have exactly one minimum at $s=0$ and no other critical points. Cusp off the region $M$ appropriately (see \cite[Figure 5]{EkholmEtnyreSullivan05a}) and define the connect sum $L_1 \# L_2$ as the union of $L_1\setminus (L_1\cap M),L_2 \setminus (L_2\cap M)$ and $\partial M.$ 
Note that the above definition depends on the choices made.

\begin{lma}[\cite{EkholmEtnyreSullivan05a}]\label{CS.lma}
Assume that $L_1,L_2,L = L_1 \# L_2$ are spin and $\dim(L)\geq 2.$
Let $(A, \pa)$ and $(A_{\mathbb{C}H_1(L)}, \pa_L)$ with grading $| \cdot|$ denote
the Legendrian contact algebras of $L$ with coefficients in
$\C$ and $\C[H_1(L)],$ respectively.
For $i = 1,2,$ let
$(A_i, \pa_i)$ and $(A_{\mathbb{C}H_1(L_i)}, \pa_{L_i})$ with grading $| \cdot|_i$ denote
the Legendrian contact algebras of $L_i$ with coefficients in
$\C$ and $\C[H_1(L_i)],$ respectively.

The connected sum construction can be set up such that $A$ (and hence
$A_{\mathbb{C} H_1(L)}$) is generated by the generators of $A_1$, $A_2$ 
and a single Reeb chord $h.$ 
Let $\rho_i: A_{\mathbb{C}H_1(L_i)} \rightarrow A_{\mathbb{C}H_1(L)}$ be the unique extension of the algebra morphism 
 $\rho_i: \C[H_1(L_i)] \rightarrow \C[H_1(L)]$ from the push-out square (\ref{PushOutCoef.eq}).

The connect sum construction can be set up such that the following holds.
\begin{enumerate}
\item
If $c\in A_i$ then $|c|_{i}=|c|$, and both $\partial(c)- \partial_{i}c$ and $\partial_L(c) - \rho_i \partial_{L_i}(c)$ are divisible by $h.$
\item
$|h|=n-1$ and $\partial_{\cdot}(h)=0$ for any of the differentials $\partial_{\cdot}.$
\end{enumerate}

\end{lma}

\begin{rmk}
In \cite{EkholmEtnyreSullivan05a}, the above lemma is proved with coefficients in $\Z_2$, but the proof generalizes to these coefficients.
Although statement (1) is originally stated in a slightly weaker version
(even ignoring the extended coefficients), the proof in 
\cite{EkholmEtnyreSullivan05a} does imply this stronger reformulation.
Moreover, Lemma \ref{CS.lma} also follows from the gradient flow tree approach
of this paper: choose the (generic) arc connecting the cusp-edges of $L_1$ and $L_2$ to be away
from all rigid gradient flow trees which define their differentials.
\end{rmk}

\begin{lma}
\label{Augmentation.lma}
There is a canonical correspondence between the set of augmentations $\epsilon$ of $(A, \pa)$ and
the set of pairs of augmentations $(\epsilon_1, \epsilon_2)$ of
$(A_1, \pa_1)$ and $(A_2, \pa_2).$
\end{lma}

\begin{proof}
We can assume the projected Reeb chord $\pi_F(h)$ sits on the hypersurface
in $N\times \R$ which separates $L_1$ and $L_2.$ 
Choose $L_1', L_2' \subset L$ to be some thickening of the two components of $L$
when decomposed by the hypersurface such that $L_1'\cap L_2' \subset \partial M.$
The DGA $(A_i', \pa_i')$ of $L_i'$ contains the sub-DGA $(A_3, \partial_3 = 0)$ generated by the chord $h.$ 
Moreover, without the isotopy $\phi_\lambda$ from Equation \ref{squeeze.eq}, the DGA $(A,\pa)$ of $L$ already has a push-out square
\[
\xymatrix{
A_3 \ar[d] \ar[r] &A_1' \ar[d]\\
A_2' \ar[r] &A }
\]
From Equation (\ref{PushOutAugmentation.eq}) we have a natural bijection between the set of augmentations $\epsilon$ of $(A, \pa)$ and
the set of pairs of augmentations $(\epsilon_1', \epsilon_2')$ of
$(A_1', \pa'_1)$ and $(A_2', \pa'_2).$ 

For any augmentation of $(A_i', \pa'_i)$, $\epsilon'_i(h) = 0,$
since $|h|_i \ne 0.$
For any other generator $c \ne h$ of $A_i',$ Lemma \ref{CS.lma} implies
\[
\pa'_i(c) = \pa_L(c) = \pa_i(c) + hw
\]
for some (possibly empty) sum of words $w.$ 
Since $\epsilon_i'(hw) = \epsilon_i'(h) \epsilon_i'(w) = 0,$ we get a natrual bijection
between augmentations $\epsilon_i'$ of $(A_i', \pa_i')$ and
$\epsilon_i$ of $(A_i, \pa_i)$ by setting $\epsilon_i(c) = \epsilon_i'(c)$ for generators $c \ne h.$
\end{proof}

\begin{thm} \label{CS.thm}
Let $(A_{\mathbb{C}H_1(L_1)}, \pa_{L_1}),(A_{\mathbb{C}H_1(L_2)}, \pa_{L_2}),(A_{\mathbb{C}H_1(L)}, \pa_{L})$ be as in Lemma \ref{CS.lma}. Then as algebraic subsets
\[
\Aug(A_{\mathbb{C}H_1(L_1)},\partial_{L_1})\times \Aug(A_{\mathbb{C}H_1(L_2)},\partial_{L_2})= \Aug(A_{\mathbb{C}H_1(L)},\partial_L).
\]
\end{thm}

\begin{proof}
We have the push-out square of commutative $\C$-algebras:

\[
\xymatrix{
&\C[H_1(L_1)]\ar@{.>}[dr]\ar[drrr]^{f_1} &&& \\
\C\ar[ur]\ar[dr] &&\C[H_1(L)]\ar[rr]|f&&\C \\
&\C[H_1(L_2)]\ar@{.>}[ur]\ar[urrr]_{f_2}&&&
}
\]

So, given any pair $(f_1,f_2)$ such that the diagram above commutes, there is a unique morphism $f : \mathbb{C}H_1(L)\rightarrow \mathbb{C}$ making the above diagram commute. Conversely, any morphism $f : \mathbb{C}H_1(L)\rightarrow \mathbb{C}$ determines a unique pair $(f_1,f_2).$ We claim that $(f_1,f_2)$ are good morphisms if and only if $f$ is also a good morphism. 
By definition, if $(f_1,f_2)$ is a good pair, then there exist augmentations $\epsilon_1$ and $\epsilon_2$ for the respective DGAs $(f_1(A_{\mathbb{C}H_1(L_1)}),f_1\partial_{L_1})$ and 
$(f_2(A_{\mathbb{C}H_1(L_2)}),f_2\partial_{L_2}).$
By Lemma \ref{Augmentation.lma}, there is a natural bijection between the set of such pairs and the set of augmentations $\epsilon: (f(A_{\mathbb{C}H_1(L)}),f\partial_{L})\rightarrow \C.$



Next note that $\Spec$ is an equivalence of categories from finitely generated $\C$-algebras to the category of algebraic subsets over $\C$, so it takes the push-out square to the pull-back square:

\[
\xymatrix{
&\Spec(\C[H_1(L_1)])\ar[dl] &&& \\
\Spec(\C) &&\Spec(\C[H_1(L)])\ar@{.>}[ul]\ar@{.>}[dl] &&\Spec(\C) \ar[ll]|f \ar[ulll]_{f_1} \ar[dlll]^{f_2} \\
&\Spec(\C[H_1(L_2)]) \ar[ul]&&&
}
\]

Since $\Spec(\C)$ is a point, it follows that $\Spec(\C[H_1(L)])\cong \Spec(\C[H_1(L_1)]) \times \Spec(\C[H_1(L_2)]).$ (The product is taken in the category of algebraic subsets, not in the category of topological spaces.)

If $X$ is the collection of good points in $\Spec(\C[H_1(L_1)])$ and $Y$ is the collection of good points in $\Spec(\C[H_1(L_2)])$, we see that $X \times Y$ is the collection of good points in $\Spec(\C[H_1(L)])\cong \Spec(\C[H_1(L_1)]) \times \Spec(\C[H_1(L_2)]).$ We claim that $Cl(X \times Y)$= $Cl(X)\times Cl(Y)$ where $Cl$ denotes the Zariski closure. To see this, we note that the subspace $\Spec(\C[H_1(L_1)]) \times {a}$ is homeomorphic to $\Spec(\C[H_1(L_1)])$ via the projection map for any $a\in Y$. 
So $(\Spec(\C[H_1(L_1)]) \times {a}) \cap Cl(X \times Y)$
is a closed set containing $(X\times{a})\cap Cl(X \times Y),$
and thus must contain $Cl(X)\times {a}.$
Taking the union with respect to $a\in Y$, we find that  $Cl(X) \times Y \subset Cl(X \times Y) $. The closed set $({b}\times \Spec(\C[H_1(L_2)]))\cap Cl(X \times Y)$ contains ${b}\times Y$ for any $b \in Cl(X).$ Taking the union with respect to $b \in Cl(X)$, we find that  $Cl(X) \times Cl(Y) \subset Cl(X \times Y)$ and since
$Cl(X) \times Cl(Y)$ is closed and contains $X \times Y,$ equality holds.
\end{proof}

The cusp connect sum can also be formed for 1-dimensional Legendrian knots, $L_1 \sharp L_2.$ 
In this case, Sivek constructs a correspondence between augmentations similar to the one in Lemma \ref{Augmentation.lma} and proves a relation for their Poincar\'e polynomials  \cite[Section 3.5]{Sivek11}: 
$P_{A, \pa ,\epsilon}(t) = P_{A_1, \pa_1 ,\epsilon_1}(t) +  P_{A_2, \pa_2 ,\epsilon_2}(t) - t.$
We can replicate Sivek's proof using iterated applications of Mayer-Vietoris sequences for $(n \ge 2)$-dimensional Legendrians to get that 
\[
P_{A, \pa ,\epsilon}(t) = P_{A_1, \pa_1 ,\epsilon_1}(t) +  P_{A_2, \pa_2 ,\epsilon_2}(t) - t^{n} \,\,\, \mbox{or} \,\,\, P_{A_1, \pa_1 ,\epsilon_1}(t) +  P_{A_2, \pa_2 ,\epsilon_2}(t) + t^{n-1}.
\]
From some simple examples, we conjecture that only the first case can occur. To eliminate the second, however, one needs to better understand in the Mayer-Vietoris sequence the image of $LCH(A_3, \pa_3=0, \epsilon_3=0)$  generated by the chord $h.$ 

\subsection{Examples}


In this subsection, we apply Proposition \ref{LCH.prop} to
construct two Legendrian surface homotopies which are not isotopies.

\begin{ex} \label{Homotopy1.ex}
Consider a Legendrian surface immersed in $\R^5$ with one transverse double point.
Let $L$ and $L'$ be $C^2$-small deformations of this immersion (support near
the double point), where the front projections
of the two branches in $\pi_F(\R^5)$ appear in Figure \ref{Homotopy1.fig}. 
Assume that the DGA of $L$ is good, that is, has an augmentation.
Since $L$ and $L'$ are Legendrian homotopic and 2-dimensional,
by \cite[Section 3]{EkholmEtnyreSullivan05a} they have the same
rotation class and Thurston-Bennequin invariant. 
Assume the Maslov class vanishes so that the grading is well-defined. 
Assume that $m\ne 1,$ where $m$ is
the number of down-cusps minus the number of up-cusps traversed, when traveling from the upper sheet to the lower sheet of either diagram.
Moreover, assume that there are no other sheets of $\pi_F(L)$ (or $\pi_F(L')$)
with $(x_1,x_2)$-coordinates nearby, as indicated in Figure \ref{Homotopy1.fig}.
Choose a small curve $h \subset N$ encircling the $(x_1,x_2)$-coordinates
of the double point. 
We imagine performing the homotopy after pinching above the curve $h$ (see the proof of Theorem \ref{VK.thm}).

\begin{figure}
 \centering
  \includegraphics[height=1.5in, trim=7.25in 0in 0in 0in]{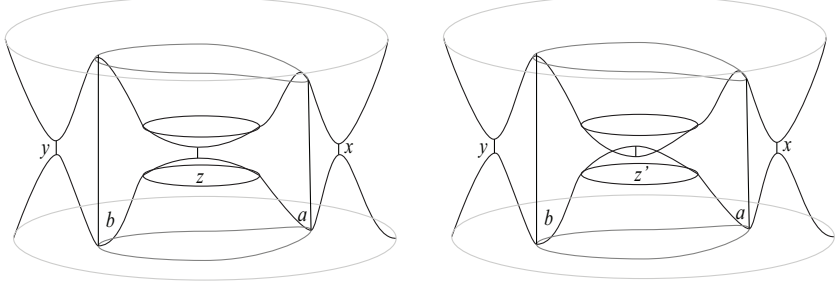}
\caption{The front projection of the two local branches
before and after the double point.  
The (front projection) of the chords $x$ and $y$ lie above the circle  $h \subset \R^2_{x_1,x_2}$ from
the proof of Theorem \ref{VK.thm}, above which the Legendrian is ``pinched."
The front projection of the Reeb chords  are denoted by vertical lines.
As Morse-Bott maxes, $b$ and $y$ are slightly longer than their counterpart Morse-Bott mins, $a$ and $x.$}
\label{Homotopy1.fig}
\end{figure}


The gradings of the chords  are
\[
|b| = 1+m, \quad |a| = |y| = m, \quad |x| = |z| = m-1, \quad |z'| = 1-m.
\]
Here we imagine a (so-called ``Morse-Bott") family of Reeb chords sitting over a circle in the $(x_1,x_2)$-plane that has been perturbed to a pair of transverse Reeb chords of relative index 1.
This applies to the pairs $(a,b)$ and $(x,y).$ See \cite[Figure 6]{EkholmEtnyreNgSullivan11}
for example.

Let $(A,\pa)$ and $(A', \pa')$ denote the DGAs of $L$ and $L'.$
 By comparing actions and using the usual gradient flow arguments from Morse-Bott perturbations, we compute the differentials to be
\begin{eqnarray*}
&\pa b = \pa'b = a +a = 0, \quad \pa y = \pa'y = x + x = 0, \quad \pa x = \pa'x = 0,& \\
& \pa a = z + x, \quad \pa'a = x, \quad \pa z = 0, \quad \pa'z' = 0.
\end{eqnarray*}
 
For any augmentation $\epsilon,$ the hypothesis $m \ne 1$ implies 
$\epsilon(z) = 0 = \epsilon(z').$
It is easy to to check that augmentations for $(A,\pa)$ and $(A', \pa')$
agree under the canonical identification of all other generators.
Moreover, under this identification of augmentations, 
\[
LCH_{1-m}(L', \pa', \epsilon) = LCH_{1-m}(L, \pa, \epsilon) \oplus \Z_2 \langle z' \rangle.
\] 
Thus $L$ and $L'$ are not Legendrian isotopic.

\end{ex}

\begin{ex} \label{Homotopy2.ex}
Consider the homotopy from Example \ref{Homotopy1.ex}, but now suppose there may be many local sheets of $\pi_F(L)$ with the same $(x_1, x_2)$-coordinates as the double point.
See Figure \ref{Homotopy2.fig}. 
Assume that $\pi_F(L)$ (and hence $\pi_F(L')$)
has no cusp-edges. For example, see the conormal construction of
the braid in \cite{EkholmEtnyreNgSullivan11}.
Note that without cusp-edges, $L$ and $L'$ have DGAs which are already augmented.

If we choose the curve $h$ (see Example \ref{Homotopy1.ex}) small enough,
we can assume all other sheets of $\pi_F(L)$ and $\pi_F(L')$ are planar in this region, so that after we pinch the Legendrians, the $\{x_2=0\}$-slice of $\pi_F(L)$ looks like Figure \ref{Homotopy2.fig} and 
$\pi_F(L')$ is a small homotopy of this similar to Figure \ref{Homotopy2.fig}.

\begin{figure}
 \centering
  \includegraphics[height=2in]{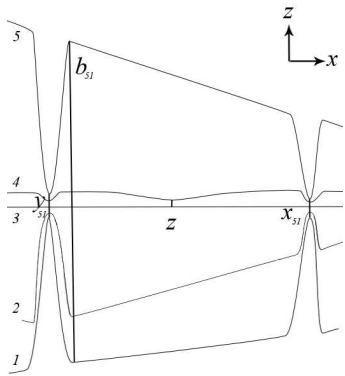}
\caption{
This is the $\{x_2=0\}$-slice of the front $\pi_F(L)$ just before the double point appears in the homotopy.
For each pair of sheets given by the graphs $f_i>f_j$ the function
$f_i - f_j$ has a max $b_{ij},$ and sitting over the curve $h,$ it has
  a saddle $y_{ij}$ and min $x_{ij}.$ 
 The pair of sheets $A=4$ and $B=3$ undergoing the homotopy have two additional
Reeb chords $a$ and $z,$ as labeled in Figure \ref{Homotopy1.fig}. Note that not all Reeb chords appear in this slice, see Figure \ref{Homotopy3.fig}.
}
\label{Homotopy2.fig}
\end{figure}

Let $(A,\pa)$ and $(A', \pa')$ denote the DGAs of $L$ and $L'.$
Because there are no cusp-edges, their gradings are
\[
|b_{ij}| = 1, \quad |y_{ij}| = 0, \quad |x_{ij}| = -1.
\]
and their differentials satisfy
\[
\pa b_{ij} = y_{ij} + Q, \quad \pa' b_{ij} = y_{ij} + Q', \quad \pa y_{ij} = \pa' y_{ij} ,
\quad \pa x_{ij} = \pa' x_{ij} 
\]
where $Q$ and $Q'$ are quadratic terms or higher.

We claim that  we can arrange  $L$ such that no word in $Q$ is purely a product of generators of type $y_{kl}.$ Similarly for $L'$ and $Q'.$

Let $(t,s = {\bf s}) \in [0,1+\epsilon] \times [0, 2\pi) $ denote the polar coordinates for the disk in $\R^2_{x_1,x_2}$ concentric with the circle $h.$
This extends a modified $t$-coordinate defined in the proof of Theorem \ref{VK.thm}, since we set the circle $h$ to be $\{t=1\}$ instead of $\{t=0\}.$ Let $A, B \in \nats$ index the two
sheets of the front defining the chord $z,$ with the $z$-coordinate of sheet $A,$ $z(A),$ greater than $z(B).$
Without loss of generality, we can assume that the disk is small enough such that $z(B) = 0$ and
$|\pa z(A)/\pa x_i| \ll 1.$ 
The explicit nature of the isotopy $\phi_\lambda,$ see Equation (\ref{squeeze.eq}), combined with the linearity of the sheets $B, k$ and $l,$ imply that the max $b_{kl}$ (resp. $b_{kB}, b_{Bl}$) and saddle $y_{kl}$  (resp. $y_{kB}, y_{Bl}$) sit above points with the same $s$-coordinate and nearby $t$-coordinates.   Moreover, the unique gradient flow tree (flow line, actually) connecting them is radial. 
Since $|  \pa z(A)/\pa x_i - \pa z(B)/\pa x_i| \ll 1,$ similar statements can be made for $b_{kA}, b_{Al}, \ldots,$ as well.
If we {\em a priori} assume that the sheets are in general position prior to the pinching isotopy, this implies that the gradient flow line connecting $b_{kl}$ to $y_{kl}$ cannot intersect the gradient flow line connecting $b_{pq}$ to $y_{pq},$  as indicated in Figure \ref{Homotopy3.fig}. This include, with a small perturbation, the $b_{AB}$ to $y_{AB}$ flow line. There is an exception: the flow line connecting $b_{kA}$ to $y_{kA}$ (resp. $b_{Al}$ to $y_{Al}$)
may intersect the one connecting the nearby $b_{kB}$ to $y_{kB}$ (resp. $b_{Bl}$ to $y_{Bl}$); however,
this will not concern us.

 Consider a contribution to $Q$ from a gradient flow tree with a positive puncture at $b_{ij}$ and $n \ge 2$ negative punctures at $y_{k_1 l_1}, \ldots,
 y_{k_n l_n}.$ Since the tree has no valence-one vertices mapping to cusp-edges, it has at least two valence-one vertices mapping to, say $y_{kl}$ and $y_{pq},$ whose edges end at a common valence-three vertex of the tree. (We consider a valence-two vertex puncture to be a valence-three vertex with a zero-length edge connecting it to a valence-one vertex puncture.) For this to happen, $l=p$ or $k=q$ and the stable manifolds of $y_{kl}$ and $y_{pq}$ must intersect. (So $y_{kl}$ and $y_{pq}$ are distinct,
 and they cannot be pairs $y_{kA}, y_{kB}$  or $y_{Al}, y_{Bl}$ which lie near each other, as in Figure \ref{Homotopy3.fig}.)  But this contradicts the above disjointness of the two flow lines connecting  $b_{kl}$ to $y_{kl}$ and $b_{pq}$ to $y_{pq},$
 since the other stable manifolds  of $y_{kl}$ and $y_{pq}$
 lie outside the disk bounded by $h.$ 
 This proves the claim for $Q.$
 
\begin{figure}
 \centering
  \includegraphics[height=3in]{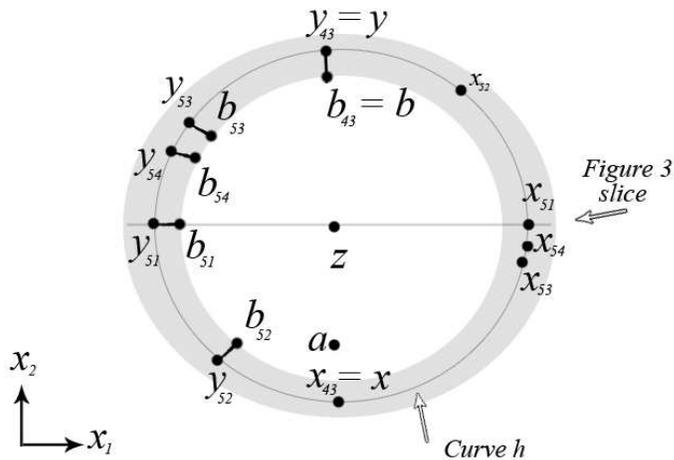}
\caption{
The shaded area, $\{1-\epsilon \le t \le 1+\epsilon\} \subset \R^2_{x_1,x_2},$  denotes the support of the isotopy $\phi_\lambda$ from Equation (\ref{squeeze.eq}). 
The horizontal line denotes the $\{x_2=0\}$-slice from Figure \ref{Homotopy2.fig}.
To avoid clutter, only chords starting at sheet 5 or running between sheets $A=4$ and $B=3$ are indicated. The gradient flow trees (lines) in $\M(b_{ij};y_{ij})$ are indicated by short solid radial lines. Here $b_{43}, y_{43}, x_{43}$ correspond to $b,y,x$ from Figure \ref{Homotopy1.fig}. Flow lines connecting $b_{54}$ to $y_{54}$ and $b_{53}$ to $y_{53}$ are close, possibly intersecting. 
}
\label{Homotopy3.fig}
\end{figure}

Computations for $\pa b, \pa a , \ldots$ are as in Example \ref{Homotopy1.ex}.
For both $\pa$ and $\pa'$ any augmentation must vanish on $b_{ij}$ and $x_{ij}$ for grading reasons. Because no term in $Q$ or $Q'$ is a word made purely of generators of the form $y_{kl},$ a 
quick algebraic check shows that the augmentation must vanish on $y_{ij}$ as the unique linear term in
a differential. 
Thus, like in Example \ref{Homotopy1.ex}, the augmentations
of $(A,\pa)$ and $(A', \pa')$ must agree, and we conclude
\[
LCH_{1}(L', \pa', \epsilon) = LCH_{1}(L, \pa, \epsilon) \oplus \Z_2 \langle z' \rangle.
\] 

\end{ex}

\end{document}